\documentclass[11pt]{article}
\usepackage{amsmath, amssymb, amsthm, mathrsfs, commath, mathtools}
\usepackage[headheight=0pt,headsep=0pt,textwidth=6in,textheight=9in]{geometry}
\usepackage[dvipsnames]{xcolor}

\usepackage{indentfirst}

\usepackage[margin=1cm,font=small]{caption}

\usepackage{tikz}
\usetikzlibrary{decorations.markings}
\usepackage{float}

\usepackage{chngcntr}

\newcommand{\R}{\mathbb{R}}

\newcommand{\C}{\mathbb{C}}

\renewcommand{\Re}{\tn{Re}}	
\renewcommand{\phi}{\varphi}

\renewcommand{\d}{\, d}									               
\newcommand{\tn}[1]{\textnormal{#1}}					                   
\newcommand{\mat}[1]{\begin{pmatrix} #1 \end{pmatrix}}	               
\renewcommand{\l}{\left}                                               
\renewcommand{\r}{\right}								               
\renewcommand{\b}[1]{\l( #1 \r)}						               
\newcommand{\mc}[1]{\mathcal{#1}}						               
\newcommand{\st}{\,\colon\,}							              	   
\newcommand{\ip}[1]{\l\langle #1 \r\rangle}				               
\renewcommand{\pd}[2]{\frac{\partial #1}{\partial #2}}                 
\newcommand{\eq}[1]{\begin{equation} #1 \end{equation}}
\newcommand{\aeq}[1]{\begin{align}\begin{split} #1 \end{split}\end{align}}

\newcommand{\Tr}{\tn{Tr}\,}

\counterwithin{equation}{section}

\theoremstyle{plain}
\newtheorem{thm}[equation]{Theorem}

\newtheorem{prop}[equation]{Proposition}

\theoremstyle{definition}
\newtheorem{defn}[equation]{Definition}
\newtheorem{exa}[equation]{Example}
\newtheorem{rmk}[equation]{Remark}

\title{A Deformation Approach to the BFK Formula}
\author{Romain Speciel}
\date{November 13, 2025}

\begin{document}
\maketitle


\abstract{Understanding how spectral quantities localize on manifolds is a central theme in geometric spectral theory and index theory. Within this framework, the BFK formula, obtained by Burghelea, Friedlander and Kappeler in 1992, describes how the zeta-regularized determinant of an elliptic operator decomposes as the underlying manifold is cut into pieces. In this paper, we present a novel proof of this result. Inspired by work of Br\"uning and Lesch on the eta invariant of Dirac operators, we derive the BFK formula by interpolating continuously between boundary conditions and understanding the variation of the determinant along this deformation.}

\section{Introduction}
\label{sec: intro}

The zeta-regularized determinant of the Laplacian, introduced in \cite{RS71}, is a global spectral invariant of smooth Riemannian manifolds with important interpretations and applications in geometry, topology and mathematical physics (see Chapter 5 of \cite{R97} for an overview and \cite{OPS88} for an elegant application). Explicitly computing this invariant is typically impossible, but a useful tool to understand its nature is the Mayer-Vietoris type gluing formula proven by Burghelea, Friedlander and Kappeler in \cite{BFK92}, commonly called the \textit{BFK formula}. The BFK formula, which applies to a broad class of elliptic boundary problems on bundles, relates the determinant of an elliptic differential operator across different complementary boundary conditions via the determinant of an intermediate lower order pseudodifferential operator on the boundary.

This formula is part of a broader body of work which seeks to understand how various global spectral quantities localize on manifolds. Such results have been successfully proved through various approaches; we provide several examples chronologically for context. In \cite{BFK92}, where the original proof of the BFK formula is presented, the authors perturb the interior operator by adding to it a multiple of the identity, then explicitly compute the corresponding variation of the determinant in terms of the perturbation parameter. In \cite{M94}, M\"{u}ller studies how the eta invariant, another natural spectral quantity which arises in index theory, varies on a manifold with cylindrical boundary as the boundary gets pushed to infinity. In \cite{MM95}, Mazzeo and Melrose compute how the eta invariant varies under analytic surgery, in which a submanifold is pushed to infinity by taking degenerating metrics which limit to a $b$-metric. These ideas are extended by Hassell, Mazzeo and Melrose in \cite{HMM95}, and applied to handle analytic torsion, a spectral invariant with connections to the combinatorial structure of the underlying manifold. In \cite{B95}, Bunke studies the adiabatic limit of the eta invariant of the Dirac operator associated to a pair of APS-type boundary conditions. In \cite{CLM96}, Cappell, Lee and Miller develop a splicing technique which is applied to study localization of low eigensolutions of a self-adjoint elliptic operator. And in \cite{BL99}, Br\"{u}ning and Lesch study the Dirac operator on a manifold with boundary, continuously deform the domain of the operator by varying the boundary conditions, and compute how the eta invariant of the Dirac operator changes through this deformation.

In this paper, we present a new proof of the BFK formula, inspired mainly by the arguments in \cite{BL99}. We develop this deformation approach with the aim of obtaining a proof method which will more readily generalize to settings where we cut along submanifolds with geometric singularities such as corners. For clarity, we state the main result here but delay precise definitions to the next section.

Let $A$ be an admissible elliptic differential operator of positive order on a compact manifold $M$ with nonempty boundary $N$. Fix $B_0$ and $B_1$ two complementary boundary conditions, and take $Q$ to be the induced correspondence operator on the boundary. Then, with $A_i$ denoting the operator $A$ confined to the domain $\ker B_i$, we have

\begin{thm}[The BFK Formula \cite{BFK92}]
\label{thm: BFK}
	The zeta-regularized determinants of $A_0$ and $A_1$ are related by
    \eq{
    \label{eq: BFK Statement}
    \frac{\det A_1}{\det A_0}=c\cdot \det Q,
    }
    where $c$ is a local quantity which can be computed in terms of the symbols of $A$, $B_0$ and $B_1$ along $N$.
\end{thm}

The proof of Theorem \ref{thm: BFK} presented here proceeds by explicitly considering how the determinant varies along a family of operators $A_t$ interpolating between $A_0$ and $A_1$, rather than by perturbing the interior operator directly, as originally done in \cite{BFK92}.

\begin{rmk}
	To provide context to this approach, we describe here in more detail the strategy employed by Br\"uning and Lesch in \cite{BL99}. The authors obtain a gluing law for the eta invariant associated to a Dirac operator on a manifold with boundary by rotating through a family of boundary conditions. More specifically, they consider a Hermitian vector bundle $E$ over a compact Riemannian manifold $M$, and a first order symmetric elliptic differential operator $D$ acting on smooth sections of the bundle (the main example of which are Dirac operators). With $N\subset M$ a compact hypersurface, they note that $D$ restricted to $C^\infty(E|_{M\setminus N})$ is no longer essentially self-adjoint, and consider therefore two natural boundary conditions: the APS boundary condition and the continuous transmission boundary condition. The authors then explicitly parametrize a path between these boundary conditions and show that the variation of the eta invariant across this path can be computed by studying the operator $D$, conjugated by appropriate rotations which map between the domains corresponding to each of the boundary conditions. (Details are outlined in Section 3 of \cite{BL99}.)
\end{rmk}

Our proof carries this idea over to the zeta-regularized determinant. The main difference is that there no longer is a natural set of rotations carrying the domain corresponding to one boundary condition to that corresponding to the other boundary condition. Instead, we work directly with the resolvent, and obtain variational formulas by studying how this operator changes as we interpolate between $B_0$ and $B_1$.

In Section \ref{sec: Setup}, we provide rigorous definitions and assumptions for our theorem, and place the BFK formula in context by describing as an example the case of the Laplacian on a compact Riemannian manifold. In Section \ref{sec: props PQR}, we prove several relationships linking the Poisson extension operator $P$, the operator $Q$ which serves as an analogue of the Dirichlet-to-Neumann map, and the resolvent $R$. In particular, we establish in Proposition \ref{prop: log of d-to-n} a surprising integral identity which later serves as the key for our proof. Finally, in Section \ref{sec: bfk}, we use properties of the trace functional and identities from the meromorphic functional calculus to deduce a link between the zeta functions of $A_0, A_1$ and the trace of the operator $Q$. The main technical challenges arise here in the subtlety of trace class operators, and this is handled by introducing high powers of regularizing operators and then applying unique continuation to deduce results globally. Finally, we use these identities to deduce the BFK formula.

\vspace{1em}
\noindent\textbf{Acknowledgements.} The author thanks Rafe Mazzeo for suggesting this project and providing kind guidance throughout. He also thanks Josef Greilhuber and Gregory Parker for many helpful conversations, and the anonymous referee for detailed and thorough feedback. This work was in part supported by an NSERC-PGSD grant.

\vspace{3em}

\section{Setup}
\label{sec: Setup}

We now describe in detail our setting, which is analogous to that of \cite{BFK92}. Let $M$ be a compact manifold of dimension $d\geq 2$ with nonempty boundary $N=\partial M$, and take $E\to M$ to be a smooth vector bundle over $M$. Suppose $A\colon C^\infty (E)\to C^\infty (E)$ is an elliptic differential operator of order $\omega>0$. Let $F_j\to N$, $1\leq j\leq k$, be smooth vector bundles over the boundary, and consider for each $j$ a pair of boundary differential operators $B^j_i \colon C^\infty(E)\to C^\infty(F_j)$, for $i=0,1$, such that $\tn{ord}\,B_1^j-\tn{ord}\,B_0^j$ is positive and constant in $j$, and $\tn{ord}\,B^j_i< \omega$. Writing
\eq{
	B_t^j=(1-t)B_0^j+tB_1^j\quad \tn{and} \quad B_t=(B_t^1,\dots, B_t^k),
}
we require that $B_0$ and $B_1$ be \textit{complementary}, meaning that the intersection of their null spaces consists of sections of $C^\infty(E)$ which vanish up to order $\omega-1$ on $N$.

Next, we assume $(A,B_t)$ is an invertible elliptic boundary problem in the sense of \cite{H07} for $0\leq t\leq 1$ and hence extends to a bijective continuous operator
\eq{
	(A,B_t)\colon H^\omega(E)\to L^2(E)\oplus H^{\omega-\omega_1(t)-1/2}(F_1)\oplus\dots \oplus H^{\omega-\omega_k(t)-1/2}(F_k),
}
where $\omega_j(t)=\tn{ord }B_t^j$. Importantly, note $\omega_j(t)$ is discontinuous at $t=0$.

\begin{exa}[The Laplacian on a manifold with boundary]
\label{ex: laplacian}
	The following setup is our main motivating example. Suppose $M$ is equipped with a Riemannian metric $g$ and denote by $\Delta_g$ the corresponding (positive) Laplacian acting on $C^\infty(M)$. Define the Dirichlet and Neumann boundary operators, denoted $\mc{D}$ and $\mc{N}$ respectively, by
	\eq{
		\mc{D}u=u|_{\partial M} \quad \tn{and}\quad \mc{N}u=\pd{u}{n}\Big{|}_{\partial M},
	}
	where $n$ is the outward-pointing unit normal at the boundary. With $E$ and $F_1$ the trivial line bundles over their respective spaces (i.e. considering the Laplacian acting simply on real valued functions), set $B_0=\mc{D}$ and $B_1=\mc{N}$. Finally, to ensure invertibility, set $A=\Delta_g+\epsilon$ for some positive $\epsilon$. The discussed assumptions are then clearly satisfied.
\end{exa}

Let $A_t$ denote the operator which acts as $A$ on its domain $\tn{dom}(A_t)=\ker (B_t)\subset H^\omega(E)$, and define the corresponding \textit{resolvent operator} by 
\aeq{
	R_t(z)\colon L^2(E)&\to H^\omega(E)\\
	u&\mapsto(A-z,B_t)^{-1}(u,0)
}
and \textit{Poisson operator} by
\aeq{
	P_t(z)\colon H^{\omega-\omega_1(t)-1/2}(F_1)\oplus\dots \oplus H^{\omega-\omega_k(t)-1/2}(F_k)&\to H^\omega(E)\\
	f=(f_1,\dots, f_k) & \mapsto (A-z,B_t)^{-1}(0,f).
}
Both of these families of operators depend meromorphically on $z$ as a consequence of the analytic Fredholm theorem.

Away from the spectrum of $A_0$, the triple $(A-z, B_0, B_1)$ gives rise to a natural map between sections of $\oplus _jF_j$ which correspond to each other as the image under either $B_0$ or $B_1$ of a solution of $(A-z)u=0$. This map is given by the operator $Q(z)=B_1P_0(z)$, which we assume to be positive and self-adjoint when $z$ is in the ray $\R_-=\{r e^{i\pi}\st r \geq 0\}$. We shall write $Q=Q(0)$.

\begin{exa}[The Dirichlet-to-Neumann map]
\label{ex: DtoN}
	We continue with Example \ref{ex: laplacian}. In that setting, the problem $(A=\Delta_g+\epsilon, B_t=(1-t)\mathcal{D}+t\mathcal{N})$ is well known to be an elliptic boundary problem (see \cite{H07} for an extensive discussion, for example). The operator $Q(z)$ described above corresponds to the standard Dirichlet-to-Neumann map on $N$ with respect to interior extension operator $\Delta_g+\epsilon -z$. In this setting, it is clear to see why the restriction $z\notin \tn{spec }A_0$ is necessary: $P_0(z)$ fails to be well defined precisely when the problem $(\Delta_g+\epsilon -z)u=0$ on $M$, $u=f$ on $N$ does not have a unique solution. Indeed, if $u,\tilde u$ are two distinct solutions to the problem, observe $u-\tilde u$ is a $z$-eigenfunction of $A_0$ (and, in the other direction, the existence of such an eigenfunction obstructs uniqueness of extensions). Next, we confirm $Q(z)$ is positive and self-adjoint on $\R_-$: by Green's identity,
	\eq{
		\ip{Q(z)f,f}_{L^2(N)}=\norm{\nabla P_0(z)f}^2_{L^2(M)}+(\epsilon-z)\norm{P_0(z)f}^2_{L^2(M)}>0
		}
	since $z<0$, and similarly
	\aeq{
		\ip{f,Q(z)g}_{L^2(N)}-\ip{Q(z)f,g}_{L^2(N)}&=\ip{B_0P_0(z)f,B_1P_0(z)g}_{L^2(N)}-\\
		&\qquad \qquad \ip{B_1P_0(z)f,B_0P_0(z)g}_{L^2(N)}\\
		&= (z-\overline z)\ip{P_0(z)f,P_0(z)g}_{L^2(M)},
	}
	which vanishes when $z\in \R$.
\end{exa}

\begin{rmk}
	Note that the computations presented above hold more generally when $A$ is a symmetric divergence-form elliptic operator, thus the assumption that $Q(z)$ is self-adjoint on $\R_-$ is automatically satisfied in those cases. In general, however, this assumption will be necessary for Proposition \ref{prop: log of d-to-n} to hold.
\end{rmk}

In order to ensure that our operators have well defined zeta-regularized determinants, we shall require that they satisfy additional spectral properties which guarantee certain path integrals in $\C$ to be well defined and convergent. This is captured by the following definition.

\begin{defn}
	For $B$ any boundary differential operator such that $(A,B)$ is an elliptic boundary problem, the angle $\theta$ is a \textit{principal angle} for $(A,B)$ if the following conditions are satisfied:
	\begin{enumerate}
		\item \textit{(Agmon angle condition)} For every $p\in M$ and $\xi\in T^*_p(M)\setminus 0$, we have
		\eq{
			\tn{spec}\,\sigma_\omega(A)(p,\xi)\cap \{r e^{i\theta}\st r \geq 0\}=\varnothing,
		}
		where $\sigma_\omega(A)$ denotes the principal symbol of $A$.
		\item \textit{(Shapiro-Lopatinski condition)} In a collar neighborhood $N\times [0,\infty)$ of the boundary, let $(x,y)$ denote the induced coordinates and $(\eta,\tau)$ the corresponding coordinates on the cotangent bundle. Write $\alpha=(\alpha',\alpha_d)$, a multi-index of length $d$, and say $A=\sum_{\alpha}a_\alpha(p)D^\alpha$ in coordinates. Then, given $r\geq 0$ and $f_j\in F_j$ for $1\leq j\leq k$, there exists a unique solution $u(x,y)$ to
		\eq{
			\sum_{\abs{\alpha}=\omega}a_\alpha(x,0)(\eta^{\alpha'})D_y^{\alpha_d}u(x,y)=re^{i\theta}u(x,y)
		}
		satisfying both $\lim_{y\to \infty}u(x,y)=0$ and
		\eq{
			\sum_{l=\omega-\omega_j}^\omega \sigma_{\omega_j}(B_j)(x,\eta)D_y^{\omega-l}u(x,0)=f_j
		}
		for every $j$.

	\end{enumerate}
\end{defn}

We assume that $\pi$ is a principal angle for $(A,B_t)$ for every $0\leq t\leq 1$, which allows us to define complex powers of $A_t$ as in the Seeley calculus. To this end, let $\gamma$ be a contour in the complex plane around the ray $\R_-$ traveling clockwise around the origin, so that it positively winds around every point of the spectrum (see Figure 1). Taking the branch cut of $z^{-s}$ along $\R_-$, define
\eq{
	A_t^{-s}=\frac{i}{2\pi}\int_{\gamma} z^{-s}R_t(z)\d z.
}
When $\tn{Re}(s)$ is sufficiently large, this operator is well defined and trace class. Its trace, denoted by $\zeta_t(s)$, extends meromorphically to the whole complex plane and is regular at $s=0$ (we refer the reader to \cite{S69-2, S69-1} for details). Finally, we may now define
\eq{
	\det A_t=e^{-\zeta_t'(0)},
}
the \textit{zeta-regularized determinant} of $A_t$.

\begin{exa}[The determinant of the Laplacian] We continue with Examples \ref{ex: laplacian} and \ref{ex: DtoN}. Since $A_t$, which is simply $\Delta_g+\epsilon$ restricted to the Robin-boundary condition domain $\ker((1-t)\mathcal{D}+t\mathcal{N}))$, is a positive self-adjoint operator and its principal symbol is  $\abs{\xi}^2$, conditions 1 and 2 in the definition of principal angle are satisfied with $\theta=\pi$. We may then define $A_t^{-s}$, and the Weyl law tells us that this will be a bounded (and in fact trace class) operator when $\Re(s)>d/2$. In fact, denoting the eigenvalues of $A_t$ by $\lambda_1,\lambda_2,\dots$, we may rewrite
\eq{
	\tn{Tr}(A_t^{-s})=\sum_j \lambda_j^{-s}
}
to then note the connection to the standard zeta function. We may then seek to take $\epsilon\to 0$ and regularize the limit in order to get a well defined determinant of the genuine Laplacian, rather than its shift $A$. This works, and is described in the proof of Theorem B* in \cite{BFK92}. See also \cite{L97} for further specialization to Laplace type operators.
\end{exa}

Next, we seek to extend our definition of zeta-regularized determinant to matrices of operators since $Q(z)$, which appears in the right hand side of Equation \ref{eq: BFK Statement}, is a $k\times k$ matrix of $\Psi$DOs. The necessary properties are described in the following definition.

\begin{defn}
	We say the operator $Q=\b{Q_{ij}}\colon \oplus_{j=1}^k C^\infty(F_j)\to \oplus_{j=1}^k C^\infty(F_j)$ is \textit{of order} $(\alpha_1,\dots, \alpha_k;\beta_1,\dots, \beta_k)$  if $\tn{ord }Q_{ij}=\alpha_i+\beta_j$. Define then the \textit{principal symbol} of $Q$ to be the matrix $\sigma(Q)=\b{\sigma_{\alpha_i+\beta_j}(Q_{ij})}$, and say $Q$ is  \textit{elliptic in the sense of Agmon, Douglis and Nirenberg} if for any $p\in M$ and $\xi\in T^*_p(M)\setminus 0$, the map $\sigma(Q)(p,\xi)$ is invertible.
\end{defn}

We shall show in Proposition \ref{prop: properties of Q} that $Q(z)$ is indeed elliptic in the sense of Agmon, Douglis and Nirenberg when $z\notin \tn{spec }A_0$. With minor modification, the notion of a principal angle carries over to such operators, and we assume $\pi$ is also a principal angle for $Q(z)$, hence the determinant of $Q(z)$ is well defined (again only when $z\notin \tn{spec }A_0$). The details of these definitions can be found in Section 2 of \cite{BFK92}.

\begin{rmk}
	For simplicity, we have taken $\pi$ to be the principal angle for all involved operators. In particular, this is the case for the discussed example of the Laplacian on a manifold, as the Dirichlet-to-Neumann operator is positive. However, all proofs presented below can be modified to the general case without too much complication: one must ensure that the path $\gamma$ used in the Seeley calculus is adapted to the principal angle, and that the dependence of the determinants on the principal angle be made explicit. The proof in \cite{BFK92}, for example, accounts for these details. Our proof, however, has an additional assumption: we make use of the fact that $Q(z)$ is self-adjoint when $z\in \R_-$ in Proposition \ref{prop: log of d-to-n}. To adapt this condition to a general principal angle, one may ask that $Q(z)$ be self-adjoint when $z$ lies on the ray corresponding to the angle.
\end{rmk}

\begin{rmk}[Decomposing the determinant along a cut]
	The original statement of the BFK formula in \cite{BFK92} differs slightly from our formulation in Theorem \ref{thm: BFK}; we outline in this remark how the traditional perspective may be recovered from our approach in the case of the Laplacian.
	
	Let $M$ be a smooth Riemannian manifold \textit{without} boundary and take $N$ an oriented submanifold of codimension $1$, which we assume for simplicity to have trivial normal bundle. Denote by $M_\tn{cut}$ the compact manifold obtained by cutting $M$ along $N$, with boundary $N^-\sqcup N^+$. Define the sum and difference operators by, respectively,
	\eq{
	\begin{gathered}
		\sigma, \delta\colon C^\infty(N^-)\oplus C^\infty(N^+)\to C^\infty(N),\\
		\sigma (f,g)=f+g,\quad  \delta(f,g)=f-g.
	\end{gathered}
	}
	Now, consider the two complementary boundary conditions $B_0=\mathcal{D}$ and $B_1=\mathcal{N}$ on $M_\tn{cut}$ as before, and denote by $B_i^{\pm}$, $i=0,1$,  their composition with the restriction map to $N^\pm$. We introduce for $\tau \in \R$ the twisted boundary operator
	\eq{
	\begin{gathered}
		T_\tau\colon C^\infty(M_\tn{cut})\to C^\infty(N^-)\oplus C^\infty(N^+)\cong C^\infty(N)\oplus C^\infty(N)\\
		u\mapsto \b{\delta(B_0^-u, B_0^+u)+\tau \cdot \delta(B_1^-u, B_1^+u), \sigma(B_1^-u, B_1^+u)}. 
	\end{gathered}
	}
	The two boundary operators $B_0$ and $T_\tau$ form a complementary pair of boundary conditions when $\tau\neq 0$.
	
	Now, following the previous examples, take $A=\Delta_g+\epsilon$ for some fixed $\epsilon>0$. Applying Theorem \ref{thm: BFK}, we obtain
	\eq{
	\label{eq: trad BFK}
		\frac{\det(A,T_\tau)}{\det(A,B_0)}= c(\tau) \det Q_\tau,
	}
	where $Q_\tau=T_\tau P_0(0)$. However, all terms in which $\tau$ appears are continuous in $\tau$. Indeed, the constant $c(\tau)$ depends on the symbol of $Q_\tau$, which varies continuously, and similarly $Q_\tau$ tends continuously to $Q_0$, the twisted Dirichlet-to-Neumann map corresponding to $M_\tn{cut}$. (The continuous dependence of the determinant on a parameter $\tau$ is discussed in particular in Section 4 of \cite{L97}.) Furthermore, $T_0$ corresponds precisely to the transmission condition across $N$, and by elliptic regularity we thus have
	\eq{
		\det(A, T_0)=\det(A)
	}
	(where the first determinant is understood on $M_\tn{cut}$, and the second on $M$). We therefore recover from Equation \ref{eq: trad BFK} the main result of \cite{BFK92}.
\end{rmk}

\begin{rmk}
	Theorem \ref{thm: BFK} is a generalization of the main result in \cite{LP05}, which considers explicitly the setting described in our series of examples. Furthermore, the proofs in both \cite{LP05} and \cite{BFK92} proceed by obtaining explicit expressions of the variation $\pd{}{t}\b{\log \det(A+t,B_1)-\log \det(A+t,B_0)}$, using the fact that for specific families of operators $X(t)$ such as $(A+t,B_0)$, $(A+t,B_1)$ and $Q(t)$, the property
	\eq{
		\pd{}{t}\log \det (X(t))=\tn{Tr}(X'(t)X^{-1}(t))
	}
	can be regularized to avoid problems with the trace class, and can be combined with properties of the trace to obtain Equation (\ref{eq: BFK Statement}). The details in the case of the Laplacian are outlined in an approachable fashion in \cite{L97}. We will instead keep $A$ unchanged on the interior and vary the boundary condition continuously from $B_0$ to $B_1$.
\end{rmk}

\vspace{3em}

\section{Properties of the operators $P$, $Q$ and $R$}
\label{sec: props PQR}

The Poisson extension operator $P$, the boundary correspondence operator $Q$ and the resolvent $R$ are closely tied to each other. In this section, we establish several fundamental properties which will later be needed to complete the proof of Theorem \ref{thm: BFK}. To begin, we study the variations of our operators.

\begin{prop}
\label{prop: derivative computations}

With $B'=B_1-B_0$, we have
    \eq{
    \pd{}{t}R_t(z)=-P_t(z)B'R_t(z)\quad\tn{and}\quad\pd{}{z}P_t(z)=R_t(z)P_t(z).
}
\end{prop}

\begin{proof}
	These are straightforward computations. For the first statement, fix $u\in C^\infty(E)$, $0<t<1$, and $z\notin \tn{spec}(A_t)$, and define
	\eq{R_{t+\epsilon}(z)u\eqqcolon v(\epsilon)=v_0+\epsilon\cdot v_1+O(\epsilon^2).}
	By definition, $B_{t+\epsilon}v(\epsilon)=0$ and $(A-z)v(\epsilon)=u$. Differentiating both of these conditions at $\epsilon=0$ yields $B_tv_1=-B'v_0$ and $(A-z)v_1=0$, from which we conclude
	\eq{v_1=-P_t(z)B'v_0,}
	and the result follows since $v_0=R_t(z)u$.
	
	The second part of the claim follows by a similar argument. This time, fix $f\in C^\infty(F_1)\oplus\dots\oplus C^\infty(F_k)$, $0\leq t\leq 1$, and $z\notin \tn{spec}(A_t)$, and define
	\eq{P_t(z+h)f\eqqcolon v(h)=v_0+h\cdot v_1+O(\,\abs{h}^2).}
	By definition, $B_tv(h)= f$ and $(A-z-h)v(h)=0$. Differentiating both of these conditions at $h=0$ yields $B_tv_1=0$ and $(A-z)v_1=v_0$, from which we conclude
	\eq{v_1=R_t(z)v_0,}
	and the result follows since $v_0=P_t(z)f$.
\end{proof}

Next, recall that in order to define the determinant of our generalized Dirichlet-to-Neumann operator
\eq{
	Q(z)\colon C^\infty(F_1)\oplus \dots \oplus C^\infty(F_k)\to C^\infty(F_1)\oplus \dots \oplus C^\infty(F_k),
}
we must first check that this system of $\Psi$DOs satisfies an adapted ellipticity condition, as explained in the introduction. The following proposition confirms this.

\begin{prop}
\label{prop: properties of Q}
	The operator $Q(z)$ is of order
	\eq{
		(-\tn{ord}\, B_0^1, \dots,-\tn{ord}\, B_0^k; \tn{ord}\, B_1^1, \dots,\tn{ord}\, B_1^k)
	} and is elliptic in the sense of Agmon, Douglis and Nirenberg.
\end{prop}
\begin{proof}
	We show the result for $Q=Q(0)$, since $A$ may be replaced by $A-z$ away from the spectrum. Identifying $N\times [0,1)$ with a neighborhood $U$ of $N\subset M$, we induce coordinates $(x,y)$, with $x\in N$ and $y\in [0,1)$, on $U$. With $D_y=-i\pd{}{y}$, we may write $A$ uniquely as
	\eq{
		A=\sum_0^\omega A_jD_y^j,
	}
	where each $A_j$ is a differential operator of order $\omega-j$ in $x$, which may be viewed as an operator on $C^\infty (E|_N)$. Let $S\colon \oplus_\omega C^\infty (E|_N)\to C^\infty (E)$ be given by
	\eq{
	S(u_1,\dots, u_\omega)=A^{-1}(u_1\otimes \delta,\dots, u_\omega \otimes \delta^{(\omega)}),
	}
	where $\delta^{(j)}=D^j_y\delta$ with $\delta$ the Dirac distribution in $y$, and $\cdot \otimes \delta^{(j)}\colon C^\infty (E|_N)\to \mathscr{D}'(E)$ maps smooth sections $u$ of $E$ over $N$ to $(\cdot \otimes\delta ^{(j)}) (u)(x,y)=u(x)\delta^{(j)}(y)$. (Note that $A^{-1}$ here may be defined by embedding $M$ into a closed manifold and extending $A$, for example.) From \cite[Theorem 2.4.i]{CP00}, $S$ is well defined and its range consists of sections which are smooth up to the boundary. Define then the operator $T\colon \oplus_\omega C^\infty (E|_N)\to \oplus_{j=1}^k C^\infty(F_j)$ by
	\eq{
		T=B_0S 
		\mat{A_1 & &\cdots & &A_\omega\\
		A_2 & &\cdots & A_\omega &0\\
		\vdots & & & & \vdots\\
		A_\omega & 0 & \cdots & & 0}.
	}
	Invertibility of the principal symbol of $T$ is in fact equivalent to the ellipticity of the boundary value problem $(A,B_0)$, following the Boutet de Monvel calculus.
	
	Now, let $P$ be a parametrix for $T$ and define $Q'=B_1SP$, which is elliptic in the sense of Agmon, Douglis and Nirenberg. Observe $Q-Q'$ is smoothing, hence $Q$ is as claimed in the statement of the proposition, and its order is as claimed from the above construction.
\end{proof}

Crucial to Section \ref{sec: bfk} is the following proposition, which defines and then relates $\log Q(z)$ to the integral of the Poisson operators.

\begin{prop}
\label{prop: log of d-to-n}
	There exists a neighborhood $U$ of the ray $\R_-$ on which $\log Q(z)$ is well defined and  holomorphic. Furthermore, when $z\in U$,
	\eq{
	\label{time int eq}
		\int_0^1B'P_t(z)\d t=\log Q(z).
	}
\end{prop}

\begin{proof} As a consequence of Proposition \ref{prop: properties of Q} and the assumption that $Q(z)$ is positive and self-adjoint when $z\in \R_-$, it follows from the Spectral Theorem that the operator $Q(z)$ has an orthonormal basis of eigenvectors with positive eigenvalues when $z\in \R_-$. We may therefore define $\log Q(z)$ by the functional calculus, with the standard branch of the logarithm. Perturbation theory of holomorphic families of operators then ensures that such a basis exists for all $z\in U$, with $U$ some neighborhood of $\R_-$ (see page 368 of \cite{K95}).

	Now, both sides of (\ref{time int eq}) are holomorphic families of operators on $U$, hence it suffices to show equality on $\R_-$ and then apply the identity theorem to deduce the result. Fix therefore some $z\in \R_-$, and suppose $f$ is a $\lambda$-eigenfunction of $Q(z)$ with $\lambda>0$. Set $u=P_0(z)f$. Since
\eq{
    B_tu=(1-t)B_0u+tB_1u=(1+t(\lambda-1))f,
}
it follows that
\eq{
    P_t(z)f=\frac{1}{(1+t(\lambda-1))}\cdot u,
}
and so
\eq{
    B'P_t(z)f=\frac{\lambda-1}{(1+t(\lambda-1))}f.
}
Finally, integrating in $t$ yields
\eq{
    \b{\int_0^1 B'P_t(z)\d t}f=\b{\int_0^1\frac{\lambda-1}{(1+t(\lambda-1))}\d t}f=\log(\lambda)f.
}
Conclude that Equation \ref{time int eq} holds for every eigenfunction, and thus holds in general, concluding the proof.
\end{proof}

\vspace{3em}

\section{The BFK Formula}
\label{sec: bfk}

We may now turn our attention to the proof of our main theorem. We begin by computing the difference of the zeta functions corresponding to $A_0$ and $A_1$ by interpolating between our boundary conditions.

\begin{prop}
\label{prop: diff of zeta functions}
	For $\Re(s)$ sufficiently large, 
	\eq{
	\label{eq: diff of zeta}
		\zeta_1(s)-\zeta_0(s)=\frac{s}{2\pi i}\int_\gamma z^{-s-1}\log \det Q(z)\d z.
	}
\end{prop}

\begin{rmk} The proposition above resembles the characterization of \textit{$\zeta$-comparable operators} introduced in Section 3.3.4 of \cite{S10}. However, our setting is different since $R_1(z)-R_0(z)$ need not be trace class, hence these operators are in fact not $\zeta$-comparable. To see this, let us return to Example \ref{ex: laplacian} and study the model case of the upper half-space $\R^n_+$ for $n\geq 3$. With
\eq{
	G(x,y,x',y')=\b{\abs{x-x'}^2+\abs{y-y'}^2}^{2-n}
}
and
\eq{
	H(x,y,x',y')=\b{\abs{x+x'}^2+\abs{y-y'}^2}^{2-n},
}
the Green's function for the Dirichlet problem is $G-H$, and for the Neumann problem is $G+H$. The difference of these is therefore $2H$, which, when written in projective coordinates with $p=x/x'$ and $q=(y-y')/x$ then restricted to the $q=0$ diagonal, equals $(x')^{2-n}(1+p)^{2-n}$. The corresponding operator is therefore not trace class, meaning that our setting does not fall into the framework of $\zeta$-comparable operators.
\end{rmk}

To address the difficulty underscored by the above remark, we achieve regularization by taking high powers of $Q(z)^{-1}$. Since $Q(z)$ is assumed to be of positive order, these powers are eventually trace class.

\begin{proof}[Proof of Proposition \ref{prop: diff of zeta functions}] It follows from \cite{S67} that, since $Q(z)$ is a positive order $\Psi$DO on $N$, a boundaryless closed manifold, the function $F(z,w)=\Tr\b{Q(z)^{-w}}$ is well defined and holomorphic in $w$ when $\Re(w)>\dim N/\tn{ord }Q(z)$. Furthermore, since $Q(z)$ depends holomorphically on $z$, we may extend $F$ meromorphically to $\C^2$. Note in particular that, by definition, $\pd{F}{w}(z,0)=-\log \det Q(z)$ since the determinant is defined precisely in term of the derivative of this meromorphic extension.

Now, take $U$ as in the proof of Proposition \ref{prop: log of d-to-n}. Both $\det Q(z)$ and $\log Q(z)$ are well defined and holomorphic on $U$, and we may  take the path $\gamma$ to lie entirely in $U$. Since the trace class is an ideal, we have that for $\Re(w)$ sufficiently large
\eq{
	\pd{F}{w}(z,w)=-\Tr\b{\log Q(z)Q(z)^{-w}}=-\int_0^1\Tr\b{B'P_t(z)Q(z)^{-w}}\d t,
}
where the second equality follows from Proposition \ref{prop: log of d-to-n}. Here, we may commute the trace and the integral as $B'P_t(z)Q(z)^{-w}$,  $0\leq t\leq 1$, is a compact path in the space of trace class operators (again because $Q(z)^{-w}$ is in the trace class).

We then integrate by parts in $z$ and apply Proposition \ref{prop: derivative computations} to obtain


\aeq{
	\frac{s}{2\pi i}\int_\gamma z^{-s-1}\pd{F}{w}(z,w)\d z&=\frac{1}{2\pi i}\int_\gamma \pd{}{z}\b{z^{-s}}\cdot \int_0^1\Tr\b{B'P_t(z)Q(z)^{-w}}\d t\d z\\
	&=\frac{-1}{2\pi i}\int_\gamma z^{-s} \int_0^1 \,\Tr\Big{(}B'R_t(z)P_t(z)Q(z)^{-w}\\
	&\qquad\qquad\qquad\qquad\quad+B'P_t(z)\pd{}{z}\b{Q(z)^{-w}}\Big{)}\d t\d z\\
	&=T_1(s,w)+T_2(s,w)
}
when $s$ is sufficiently large so that the terms at infinity tend to zero. Here $T_1(s,w)$ and $T_2(s,w)$ represent the two terms of this expression. By commutativity of the trace (and once again thanks to the $Q(z)^{-w}$ term ensuring our integrand is trace class),
\eq{
	T_1(s,w)=\frac{-1}{2\pi i}\Tr\b{\int_\gamma z^{-s}\int_0^1 P_t(z)Q(z)^{-w}B'R_t(z) \d t \d z}.
}
We note, however, that this expression is regular at $w=0$, and deduce that for $s$ with sufficiently large real part (and hence for all $s$ by analytic continuation), Proposition \ref{prop: derivative computations} gives
\aeq{
	T_1(s,0)&=\Tr\b{\frac{1}{2\pi i}\int_\gamma z^{-s}\int_0^1\pd{}{t}R_t(z)\d t\d z}\\
	&=\Tr\b{\frac{i}{2\pi}\int_\gamma z^{-s}\b{R_0(z)-R_1(z)}\d z}\\
	&=\Tr\b{A_0^{-s}-A_1^{-s}}=\zeta_0(s)-\zeta_1(s).
}
On the other hand, $T_2$ vanishes at $w=0$ since $\pd{}{z}\b{Q(z)^0}=0$. We conclude
\eq{
	\zeta_1(s)-\zeta_0(s)=\frac{-s}{2\pi i}\int_\gamma z^{-s-1}\pd{F}{w}(z,0)\d z=\frac{s}{2\pi i}\int_\gamma z^{-s-1}\log \det Q(z)\d z,
}
as desired.
\end{proof}

The final step towards our main theorem is to apply the well known asymptotic expansion of determinants along a ray in the complex plane, proven in the appendix of \cite{BFK92} and stated for $\log \det Q(x)$ below.

\begin{prop} 
\label{prop: asymptotic expansion}
	The function $\log\det  Q (x)$ admits an asymptotic expansion of the form
	\eq{
		\log\det  Q (x)\sim \sum_{j=-(d-1)}^\infty \pi_j\abs{x}^{-j/2}+\sum_{j=0}^{d-1} q_j\abs{x}^{j/2}\log \abs x \quad\tn{as }x\to -\infty.
	}
	The coefficients $\pi_j$ and $q_j$ can be evaluated in terms of the symbol of $Q$, and thus depend only on the germs of the symbols of $A, B_0$ and $B_1$ at $N$.
\end{prop}

We may now prove our theorem.

\begin{proof}[Proof of Theorem \ref{thm: BFK}]
	 We split the integral on the right hand side of Equation (\ref{eq: diff of zeta}) into three pieces by splitting the contour $\gamma\subset U$ into the ray $\gamma_{+,\delta}$ a distance $\delta/2$ above the negative real axis, a small loop $\gamma_{\epsilon,\delta}$ of radius $\epsilon$ around the origin, and the ray $\gamma_{-,\delta}$ a distance $\delta/2$ below the negative real axis, as in Figure 1.

\begin{figure}[H]
\centering

\begin{tikzpicture}

    \draw[thick] (-3,0.2) -- (0.5,0.2);
    \draw[thick] (-3,-0.2) -- (0.5,-0.2);
    \draw[thick] (1,0) circle (0.54cm);
    \filldraw[white] (-3,-0.18) rectangle (1,0.18);
    \draw[thick, ->] (-2.001,0.2) -- (-2,0.2);
    \draw[thick, ->] (-2.05,-0.2) -- (-2.1,-0.2);

    \draw[help lines, ->] (-3,0) -- (3,0);
    \draw[help lines, ->] (1,-3) -- (1,2);
    \draw[<->] (0.2,-0.18) -- (0.2,0.18);
    \draw[<->] (1.01,0.01) -- (1.37,0.37);
    
    \filldraw[white] (0.28,-0.1) rectangle (0.38,0.1);
    \filldraw[white] (-2.5,-3) rectangle (2.5,-1.5);
    
    \node at (2.5,1.5) {$\C$};
    \node at (1.7,0.5) {$\gamma_{\epsilon,\delta}$};
    \node at (-1,0.5) {$\gamma_{+,\delta}$};
    \node at (-1,-0.5) {$\gamma_{-,\delta}$};
    \node at (0.35,0.0) {$_{_\delta}$};
    \node at (1.3,0.13) {$_{_\epsilon}$};
    \node[label={[align=center]below:Figure 1: The contour $\gamma$ used in\\the proof of Theorem \ref{thm: BFK}.}] at (0,-1.5) {};
    
\end{tikzpicture}
\end{figure}

	\noindent Then, take $\delta\to 0$ to obtain
	\aeq{
	\label{almost done}
		\frac{s}{2\pi i}\int_\gamma z^{-s-1}\log \det Q(z)\d z=\frac{s}{\pi}\sin (\pi s)&\int_{-\infty}^{-\epsilon}(-x)^{-s-1}\log\det  Q (x)\d x\\
		&+\frac{s}{2\pi i}\int_{\gamma_{\epsilon,0}} z^{-s-1}\log \det Q(z)\d z.
	}
	By Proposition \ref{prop: asymptotic expansion}, we compute, initially for $\Re(s)>(d-1)/2$ but then by analytic continuation to the right half plane $\Re(s)>-(N+1)/2$ for any $N\geq 0$, that
	\eq{
	\int_{-\infty}^{-\epsilon}(-x)^{-s-1}\log\det  Q (x)\d x=\sum_{j=-(d-1)}^N \frac{\pi_j}{s+j/2}+\sum_{j=0}^{d-1}\frac{q_j}{(s-j/2)^2}+h(s),
	}
	where $h(s)$ is holomorphic when $\Re(s)>-(N+1)/2$. Therefore,
	\eq{
		\pd{}{s}\bigg|_{s=0}\b{\frac{s}{\pi}\sin (\pi s)\int_{-\infty}^{-\epsilon}(-x)^{-s-1}\log\det  Q (x)\d x}=\pi_0.
	}
	To address the second term of Equation \ref{almost done}, recall that $\det Q(z)$ is holomorphic on $U$ and so, by the residue theorem,
	\aeq{
		\pd{}{s}\bigg|_{s=0}\b{\frac{s}{2\pi i}\int_{\gamma_{\epsilon,0}} z^{-s-1}\log \det Q(z)\d z}&=\frac{1}{2\pi i}\int_{\gamma_{\epsilon,0}}\frac{1}{z}\log \det  Q (z)\d z\\
		&=-\log \det Q.
	}
	(Note the negative sign above comes from the fact that the loop $\gamma_{\epsilon,0}$ is travelled clockwise.) Combine these observations to conclude
	\aeq{
		\log \det A_1-\log \det A_0&=-\pd{}{s}\bigg|_{s=0}\b{\zeta_1(s)-\zeta_0(s)}\\
		&=-\pd{}{s}\bigg|_{s=0}\b{\frac{s}{2\pi i}\int_\gamma z^{-s-1}\log \det Q(z)\d z}\\
		&=\log \det Q-\pi_0.
	}
	The result follows after exponentiating.
\end{proof}

\newpage

\bibliography{bib.bib}{}

@article{BFK92,
title = {Meyer-{V}ietoris type formula for determinants of elliptic differential operators},
journal = {Journal of Functional Analysis},
volume = {107},
number = {1},
pages = {34-65},
year = {1992},
issn = {0022-1236},
doi = {https://doi.org/10.1016/0022-1236(92)90099-5},
url = {https://www.sciencedirect.com/science/article/pii/0022123692900995},
author = {Burghelea, D. and Friedlander, L. and Kappeler, T.}
}

@article{LP05,
doi = {10.1088/0305-4470/38/41/009},
url = {https://doi.org/10.1088/0305-4470/38/41/009},
year = {2005},
publisher = {},
volume = {38},
number = {41},
pages = {8967},
author = {Loya, P. and Park, J.},
title = {$\zeta$-determinants of Laplacians with Neumann and Dirichlet boundary conditions},
journal = {Journal of Physics A: Mathematical and General}
}

@article{BL99,
author = {Br{\"u}ning, J. and Lesch, M.},
title = {On the $\eta$-invariant of certain nonlocal boundary value problems},
volume = {96},
journal = {Duke Mathematical Journal},
number = {2},
publisher = {Duke University Press},
pages = {425 -- 468},
year = {1999},
doi = {10.1215/S0012-7094-99-09613-8},
URL = {https://doi.org/10.1215/S0012-7094-99-09613-8}
}

@article{RS71,
title = {R-Torsion and the {L}aplacian on {R}iemannian manifolds},
journal = {Advances in Mathematics},
volume = {7},
number = {2},
pages = {145-210},
year = {1971},
issn = {0001-8708},
doi = {https://doi.org/10.1016/0001-8708(71)90045-4},
url = {https://www.sciencedirect.com/science/article/pii/0001870871900454},
author = {Ray, D.B. and Singer, I.M.}
}

@book{K95,
	author = {Kato, T.},
	title = {Perturbation Theory for Linear Operators},
	year = 1995,
	publisher = {Springer Berlin, Heidelberg}
}

@book{H07,
	author = {H\"ormander, L.},
	title = {The Analysis of Linear Partial Differential Operators III},
	year = 2007,
	publisher = {Springer Berlin, Heidelberg}
}

@book{S10,
    author = {Scott, S.},
    isbn = {9780198568360},
    title = {Traces and Determinants of Pseudodifferential Operators},
    publisher = {Oxford University Press},
    year = {2010},
    doi = {10.1093/acprof:oso/9780198568360.003.0001},
    url = {https://doi.org/10.1093/acprof:oso/9780198568360.003.0001},
    eprint = {https://academic.oup.com/book/0/chapter/156239007/chapter-pdf/38996734/acprof-9780198568360-chapter-1.pdf},
}

@book{CP00,
    author = {Chazarain, J. and Piriou, A.},
    isbn = {9780080875354},
    title = {Introduction to the Theory of Linear Partial Differential Equations},
    publisher = {North-Holland Publishing Company},
    year = {2000}
}

@article{S69-2,
 ISSN = {00029327, 10806377},
 URL = {http://www.jstor.org/stable/2373312},
 author = {Seeley, R.T.},
 journal = {American Journal of Mathematics},
 number = {4},
 pages = {963--983},
 publisher = {Johns Hopkins University Press},
 title = {Analytic Extension of the Trace Associated with Elliptic Boundary Problems},
 urldate = {2024-04-26},
 volume = {91},
 year = {1969}
}

@article{S69-1,
 ISSN = {00029327, 10806377},
 URL = {http://www.jstor.org/stable/2373309},
 author = {Seeley, R.T.},
 journal = {American Journal of Mathematics},
 number = {4},
 pages = {889--920},
 publisher = {Johns Hopkins University Press},
 title = {The Resolvent of an Elliptic Boundary Problem},
 urldate = {2024-04-26},
 volume = {91},
 year = {1969}
}

@article{MM95,
 URL = {https://doi.org/10.1007/BF01928215},
 author = {Mazzeo, R. and Melrose, R.},
 journal = {Geometric and Functional Analysis},
 number = {5},
 pages = {14--75},
 title = {Analytic surgery and the eta invariant},
 year = {1995}
}

@article{HMM95,
 URL = {https://www.intlpress.com/site/pub/files/_fulltext/journals/cag/1995/0003/0001/CAG-1995-0003-0001-a004.pdf},
 author = {Hassell, A. and Mazzeo, R. and Melrose, R.},
 journal = {Communications in Analysis and Geometry},
 volume = {3},
 number = {1},
 pages = {115--222},
 title = {Analytic surgery and the accumulation of eigenvalues},
 year = {1995}
}

@article{CLM96,
 URL = {https://www.maths.ed.ac.uk/~v1ranick/papers/clm1.pdf},
 author = {Cappell, S. and Lee, R. and Miller, E.Y.},
 journal = {Communications on Pure and Applied Mathematics},
 volume = {49},
 pages = {825--866},
 title = {Self-Adjoint Elliptic Operators and Manifold Decompositions Part {I}: Low Eigenmodes and Stretching},
 year = {1996}
}

@article{M94,
 author = {M\"{u}ller, W.},
 journal = {Journal of Differential Geometry},
 volume = {40},
 pages = {311--377},
 title = {Eta Invariants and Manifolds with Boundary},
 year = {1994}
}

@article{L97,
 author = {Lee, Y.},
 journal = {Differential Geometry and its Applications},
 volume = {7},
 pages = {325--340},
 title = {Mayer-{V}ietoris formula for determinants of elliptic operators of {L}aplace-{B}eltrami type (after {B}urghelea, {F}riedlander and {K}appeler)},
 year = {1997}
}

@article{B95,
 author = {Bunke, U.},
 journal = {Journal of Differential Geometry},
 volume = {41},
 pages = {397--448},
 title = {On the gluing problem for the $\eta$-invariant},
 year = {1995}
}

@article{S67,
 author = {Seeley, R.T.},
 journal = {Proceedings of Symposia in Pure Mathematics},
 pages = {288--307},
 publisher = {American Mathematica Society},
 title = {Complex Powers of an Elliptic Operator},
 year = {1967}
}

@article{OPS88,
title = {Extremals of determinants of Laplacians},
journal = {Journal of Functional Analysis},
volume = {80},
number = {1},
pages = {148-211},
year = {1988},
issn = {0022-1236},
doi = {https://doi.org/10.1016/0022-1236(88)90070-5},
url = {https://www.sciencedirect.com/science/article/pii/0022123688900705},
author = {B. Osgood and R. Phillips and P. Sarnak}
}

@book{R97, 
place={Cambridge},
series={London Mathematical Society Student Texts},
title={The Laplacian on a Riemannian Manifold: An Introduction to Analysis on Manifolds},
publisher={Cambridge University Press},
author={Rosenberg, S.}, year={1997},
collection={London Mathematical Society Student Texts}
}
\bibliographystyle{plain}

\end{document}